\documentclass{amsart}
\usepackage{amsthm,amsmath,amsfonts,amssymb,amscd,mathrsfs,graphics}
\usepackage{latexsym}
\usepackage{graphicx}

\usepackage{epic}

\usepackage{enumerate}
\usepackage[...]{youngtab}

\usepackage{tikz}

\newtheorem{thm}{Theorem}[section]

\newtheorem{lem}[thm]{Lemma}
\newtheorem{prop}[thm]{Proposition}
\newtheorem{cor}[thm]{Corollary}

\makeatletter
\renewcommand{\@seccntformat}[1]{\S{\csname
the#1\endcsname}\hspace{0.5em}}
\makeatother

\begin{document}

\title{Difference sets disjoint from a subgroup }

\author{ Courtney Hoagland, Stephen P.  Humphries, Seth Poulsen}
  \address{Department of Mathematics,  Brigham Young University, Provo, 
UT 84602, U.S.A.
E-mail: courtneyh24601@gmail.com, steve@mathematics.byu.edu, poulsenseth@yahoo.com}
\date{}
\maketitle

\begin{abstract}  
We study finite groups $G$ having a subgroup $H$ and $D \subset G \setminus H$ such that the multiset $\{ xy^{-1}:x,y \in D\}$ has every non-identity element occur the same number of times (such a $D$ is called a {\it difference set}). We show that $H$ has to be normal, that $|G|=|H|^2$, and that $|D \cap Hg|=|H|/2$  for all $g \notin H$. We show that $H$ is contained in every normal subgroup of prime index, and other properties. We give a $2$-parameter family of examples of such groups.  We show that such groups have Schur rings with four principal sets. 
\medskip

\medskip
\end{abstract}

\theoremstyle{plain}

\theoremstyle{definition}
\newtheorem*{dfn}{Definition}
\newtheorem{exa}[thm]{Example}
\newtheorem{rem}[thm]{Remark}

\newcommand{\ds}{\displaystyle}
\newcommand{\bs}{\boldsymbol}
\newcommand{\mb}{\mathbb}
\newcommand{\mc}{\mathcal}
\newcommand{\mf}{\mathfrak}
\renewcommand{\mod}{\operatorname{mod}}
\newcommand{\mult}{\operatorname{Mult}}

\def \a{\alpha} \def \b{\beta} \def \d{\delta} \def \e{\varepsilon} \def \g{\gamma} \def \k{\kappa} \def \l{\lambda} \def \s{\sigma} \def \t{\theta} \def \z{\zeta}

\numberwithin{equation}{section}

\setlength{\leftmargini}{1.em} \setlength{\leftmarginii}{1.em}
\renewcommand{\labelenumi}{\setlength{\labelwidth}{\leftmargin}
   \addtolength{\labelwidth}{-\labelsep}
   \hbox to \labelwidth{\theenumi.\hfill}}

\maketitle

 \section{Introduction}
 
 For a group $G$ we will identify a finite subset $X \subseteq G$ with the element $\sum_{x \in X} x \in \mathbb Q G$ of the group algebra. We also let $X^{-1}=\{x^{-1}:x \in X\}$.  Also, write $\mathcal C_n$ for the cyclic group of order $n$.

 A $(v,k,\lambda)$ {\it difference set} is a subset $D \subset G, |D|=k$, where $G$ is a  finite group such that every element $1 \ne g \in G$ occurs $\lambda $ times in the multiset $\{xy^{-1}:x,y \in D\}$. Further, $|G|=v$.
 
 It is well-known that if $D \subset G$ is a difference set, then $gD=\{gd:d \in D\}$ and $\alpha(D)$ are also difference sets, for any $g \in G, \alpha \in \mathrm{Aut}(G)$. 
 Thus in some sense, difference sets are spread out evenly over the group $G$. In this paper we seek to restrict the types of difference sets considered by imposing the following conditions:

We assume that  $D \subset G$ is a $(v,k,\lambda)$ difference set  
where is a subgroup $1\ne H \le G$ and $m \ge 0$ such that

\noindent (1) $D \cap D^{-1}=Hg_1\cup \dots \cup Hg_m$;

\noindent (2) $G \setminus (D \cup D^{-1})=H \cup Hg_1'\cup \dots \cup Hg_m'$.

Here $H, Hg_1,\dots, Hg_m, Hg_1',\dots, Hg_m'$ are distinct cosets of $H$.
Let $$ h=|H|,\quad u=|G:H|.$$ Then we have $h>1$. 
A group having a difference set of the above type will be called a $(v,k,\lambda)_m$ {\it relative skew Hadamard difference set group} (with difference set $D$ and subgroup $H$).

Recall that a group $G$ has a {\it skew Hadamard difference set} if it has a difference set $D$ where $G=D\cup D^{-1} \cup \{1\}$ and $D \cap D^{-1}=\emptyset$. Such groups  have been studied in \cite {cf,co,dy,dw,ev,fx,im,cg}.

Our first result is
\begin{thm} \label{thm1}  Let $G$ be a $(v,k,\lambda)_m$  relative skew Hadamard difference set group with subgroup $H$ and difference set $D$. Then 

\noindent (i) $ h=u$ is even, $v=|G|=h^2$, and  
$$\lambda=\frac 1 4  {h(h-2)},\,\,\,  k=\frac 1 2 h(h-1).$$

Now assume that $m=0$. Then we have:

\noindent (ii) $H\triangleleft G$;

\noindent (iii) each non-trivial coset  $Hg\ne H$ meets $D$ in $h/2$ points;

\noindent (iv) $H$ contains the subgroup generated by all the involutions in $G$;

\noindent (v)  the subgroup $H \le G$ does not have a complement.
\end{thm}

Let $\Phi(G)$ be the Frattini subgroup of $G$, the intersection of all the maximal subgroups of $G$.
We also have the following result which concerns maximal subgroups of $G$:
  \begin{thm} \label {thmfr}      
  Let $G$ be a group that is a $(v,k,\lambda)_0$  relative skew Hadamard difference set group with subgroup $H$ and difference set $D$.  Then
  
   (a) 
  (i)  every index $2$ subgroup of $G$ contains $H$ and $D$ meets each such  subgroup in exactly $\lambda$ points.   
  
  (ii) if  $N \triangleleft G$ has odd prime index $p$, then $H \le N$. In this situation each non-trivial coset of $N$ meets $D$ in $\frac {1}{2p} h^2$ elements, while $|N \cap D|=\frac {1}{2p} h(h-p)$.

\noindent (b)  Now assume that $G$ is also a $2$-group.
  Then $H \le \Phi(G)$.
  Further, $D$ meets each maximal subgroup of $G$  in exactly $\lambda$ points.
  \end{thm}
  

Our original motivation for studying  $(v,k,\lambda)_0$   relative skew Hadamard difference set groups was to produce examples of Schur rings with a small number of principal sets.

We now define Schur rings \cite {mp2,sch,wie,w1}.

A subring $\mathfrak S$ of the group algebra $\mathbb C G$ is called a \emph{Schur ring} (or S-ring) if there is  a partition  $\mathcal{K}=\{C_i\}_{i=1}^r$ of $G$ such that the following hold:
\begin{enumerate}
\item $\{1_G\}\in \mathcal{K}$;
\item for each $C\in \mathcal{K}$, $C^{-1}\in\mathcal{K}$;
\item ${C_i}\cdot  {C_j}=\sum_k \lambda_{i,j,k}  {C_k}$; for all $i,j \le r$.
\end{enumerate}

The $C_i$ are called the {\it principal sets} of $\mathfrak S$. Then we have:

  \begin{thm} \label{thmsring}
 Let $G$ be a  $(v,k,\lambda)_0$  relative skew Hadamard difference set group with 
  difference set $D$  and subgroup $H$. Then
$$\{1\}, \,\, H \setminus \{1\},\,\,D,\,\, D^{-1},$$ are the principal sets of a commutative  Schur-ring over $G$.\end{thm}

Each Schur ring is also an association scheme over $G$.
The $P$-matrix (where the columns correspond to the generators $1,H-1,D,D^{-1}$ and rows correspond to representations/eigenvalues of the Schur ring) of the  association scheme corresponding to the above Schur ring is 
$$\begin{pmatrix}
1 &  h - 1 &  \frac 1 2  h(h-1)  & \frac 1 2  h(h-1)\\
1  & h - 1  & -\frac 1 2  h   & -\frac 1 2  h \\
1  & -1  & -\frac 1 2  ih   & \frac 1 2  ih \\
1 &  -1 &  \frac 1 2  ih   & -\frac 1 2  ih 
\end{pmatrix}
$$
 Here $i^2=-1$. 
 Theorem \ref {thmsring} allows us to show
   \begin{thm} \label{thmminpoly}
 Let $G$ be a  $(v,k,\lambda)_0$   relative skew Hadamard group with 
  difference set $D$  and subgroup $H$. Then
the minimal polynomial for $D$ is
$$ \mu(D)=(x-k)\left (x+\frac h 2\right  )\left (x^2+\frac {h^2} 4\right ).
$$
Further, the eigenvalues $k, -h/2,ih/2,-ih/2$ have multiplicities
$$1,\,\,\,h-1,\,\,\, h(h-1)/2,\,\,, h(h-1)/2  $$ (respectively).
\end{thm}

One can say something about the image of $D$ under an irreducible representation of $G$:

\begin{thm}  \label {thmrep} Let $G$ be a  $(v,k,\lambda)_0$  difference set group with 
  difference set $D$  and subgroup $H$.
Let $\rho$ be a non-principal irreducible representation of $G$ of degree $d$. Then $\rho(G)=0I_d, \rho(D^{-1})=\rho(D)^*$ and we have one of the following:

\noindent (i) $\rho(H)=0I_d$ and  $\rho(D)=\mathrm{diag}\left (\varepsilon_1 i\frac h 2,\varepsilon_2  i\frac h 2,\dots,\varepsilon_d  i\frac h 2\right ),$
for some  $\varepsilon_i \in \{-1,1\}$; 

\noindent (ii)  $\rho(H)=hI_d$ and  $\rho(D)=-\frac h 2 I_d$.
\end{thm}

  We next  give examples of families of $(v,k,\lambda)_0$ relative skew Hadamard difference set groups.
Let $n \ge 2, 0 \le k<n-1$ and define the following bi-infinite family of groups:
\begin{align*}
\mathfrak G_{n,k}=\langle& a_1,\dots,a_n,b_1,\dots,b_n|
a_i^2=b_{i+k}, 1\le i\le n, (\text{indices taken mod $n$}),      \\
&a_{2}^{a_1}=a_{2}b_1,
a_{3}^{a_1}=a_{3}b_2,\dots,a_{k+1}^{a_1}=a_{k+1}b_k,\\&\qquad 
   (a_1,a_{k+2})=(a_1,a_{k+3})=\dots=(a_1,a_n)=1,\\&\qquad 
   (a_i,a_j)=1, \text{ for } 1<i,j\le n,\\&\qquad
   \text{ and } b_1,\dots,b_n \text { are central involutions}
   \rangle.
\end{align*}
We will show:
\begin{thm} \label{thmgnk} For $n \ge 2, 0 \le k<n-1,$ the group $\mathfrak G_{n,k}$ is a relative skew Hadamard difference set group.
\end{thm}

Next, we note that, thinking of $D$ as a $h^2 \times h^2$ matrix via the regular representation, the matrix 
$$\mathcal H(D)=2D-J_{h^2}$$ is a matrix with entries $\pm 1$. In fact we have
\begin{thm} \label{thmhad} The matrix $\mathcal H(D)$ is a Hadamard matrix and its minimal polynomial is 
$$(x+h)(x^2+h^2).$$
\end{thm}

We  recall that an  $n \times n$ matrix $M$ is a {\it Hadamard matrix} if its entries are $\pm 1$ and $MM^T=nI_n.$ \medskip

We note that the difference sets that we study satisfy the parameter condition given by Kesava Menon in \cite {km}, and so (in this case, their complements) are examples of what are known as Menon difference sets.   Thus the groups $\mathfrak G_{n,k}$ determine  a $2$-parameter family of Menon difference sets.


\medskip

\noindent {\bf Acknowledgements} We are grateful to Pace Nielsen for useful conversations regarding this paper. All calculations made in the preparation of this paper were accomplished using Magma \cite {Ma}. 

\section{Results concerning the parameters}

In this section we prove Theorem \ref {thm1} (i): 

\begin {lem} \label{lem1}
Let $G$ be a $(v,k,\lambda)_m$   relative skew Hadamard group with 
  difference set $D$  and subgroup $H$. Then
   $h=u$ is even. Further we have $|G|=h^2$ and 
$$\lambda=\frac 1 4 {h(h-2)},\quad k=\frac 1 2 h(h-1).$$
\end{lem}
\noindent {\it Proof}
Let $$A=Hg_1\cup \dots \cup Hg_m, \quad B=Hg_1'\cup \dots \cup Hg_m',$$ and 
$D=A+D_1, D^{-1}=A+D_1^{-1},$ where $A \cap D_1=\emptyset$. Thus we have 
$$|A|=|B|=hm,\quad  |D|=k=hm+|D_1|.$$
 Then from (1) and (2) of $\S 1$ we obtain
$$G=H+B+D_1+A+D_1^{-1}.$$ 
Thus we have $$v=|G|=h+hm+|D_1|+hm+|D_1^{-1}|=h+2hm+2|D_1|=h+2k.$$

Solving $v=hu, k(k-1)=\lambda(v-1), v=h+2k$
gives
$$\lambda=\frac 1 4 \frac {(hu-h)(hu-h-2)}{hu-1}.
$$

Let 
\begin{align*}
\tag {2.1} \label{t21}
a=\gcd(hu-h,hu-1), \,\,\,\, b=\gcd(hu-h-2,hu-1).
\end{align*}
 Then one can see that
$$a=\gcd(h-1,u-1),\quad b=\gcd(h+1,u+1).$$
Thus $\gcd(a,b) |2$ since $a|(h-1), b|(h+1)$ and $h > 1$. 
 
We wish to show that $h=u$.
 
Now if we have $h < u$, then we cannot have 
$(u+1)|(h+1)$,
so that  we have  $ab \le (h-1)(u+1)/2$. This gives  
$$ab \le \frac 1 2 (h-1)(u+1)=\frac 1 2 (hu-1+h-u)<\frac 1 2 (hu-1).$$ 
While if $h>u$, then  we cannot have   
$(h+1)|(u+1)$, so that       $ ab\le (u-1)(h+1)/2$, giving
$$ab\le \frac 1 2 (u-1)(h+1)=\frac 1 2 (hu-1+u-h) <\frac 1 2 (hu-1).
$$ Thus in both cases we get
 $ab <\frac 1 2 (hu-1).$ We show how this gives a contradiction. There are two cases:
 
 \noindent {\bf Case 1:} $\gcd(a,b)=1$. Then $a|(hu-1), b|(hu-1)$ and $\gcd(a,b)=1$ gives $ab|(hu-1)$. So let $hu-1=abc, c \in \mathbb N$. Then
 from (\ref{t21}) we have 
  $$\gcd\left ( \frac {hu-h}{a},c\right )=\gcd\left (
 \frac {hu-h-2}{b},c\right )=1,$$
 so that   
 \begin{align*}
 &\lambda=\frac 1 4 \frac {(hu-h)(hu-h-2)}{hu-1}=
 \frac 1 4 \frac {(hu-h)}{a}\frac {(hu-h-2)}{b}\frac {1}{c},
 \end{align*} which implies that $c=1$. But then we have $ab=hu-1>\frac {hu-1}{2}$, a contradiction.

 \noindent {\bf Case 2:}  $\gcd(a,b)=2$. Then $(ab/2)|(hu-1)$ and 
 so that  $hu-1=\frac {ab} {2} c, c \in \mathbb N$, where $\gcd\left ( \frac {hu-h}{a},c\right )=\gcd\left (
 \frac {hu-h-2}{b},c\right )=1$. Then 
  \begin{align*}
 &\lambda=\frac 1 4 \frac {(hu-h)(hu-h-2)}{hu-1}=
  \frac 1 2 \frac {(hu-h)}{a}\frac {(hu-h-2)}{b}\frac {1}{c}.
 \end{align*}
 This again forces $c=1$, so that $\frac {ab}2=hu-1,$ which is a contradiction. Thus $h=u$ and $v=h^2$. 
 
  Now if $h=u$, then we have 
  $$\lambda=\frac {(h^2-h)(h^2-h-2)}{4(h^2-1)}=\frac {h(h-1)(h-2)(h+1)}{4(h-1)(h+1)}=\frac {h(h-2)}{4},$$ so that $h$ is even.
 \qed\medskip

\section {$H$ is normal}

Let $D$ denote the difference set where $G=D \cup D^{-1}\cup H, H \le G, D \cap H=D \cap D^{-1}=\emptyset$.

Order the elements of $G$ according to the cosets $Hg_1,Hg_2,\dots,Hg_h$.

Then thinking of $D$, $H$ and $G$ as matrices via the regular representation (relative to the above order of $G$) we have
\begin{align*}
\tag {3.1} 
G=D+D^{-1}+H,\qquad D\cdot D^{-1}=\lambda G+(k-\lambda)\cdot 1.
\end{align*}
Note that the fact that $D^{-1}$ is also a difference set \cite [p. 57]  {mp}, together with the last equation of (3.1), gives $DD^{-1}=D^{-1}D$. 

For $m \in \mathbb N$ let $J_m$ be the all $1$ matrix of size $m \times m$. 
Then we have ordered the elements of $G$ so that $H=\mathrm{diag} (J_h,J_h,\dots,J_h)$. So solving for $D^{-1}$ from the first equation of (3.1), and using $DG=kG$,  the second equation gives
\begin{align*} \tag {3.2}
(k-\lambda)(G-1)=D^2+DH.
\end{align*}
However  (since $D^{-1}$ is also a difference set) we can interchange $D$ and $D^{-1}$ so as to obtain
\begin{align*}
\tag {3.3}
(k-\lambda)(G-1)=(D^{-1})^2+D^{-1}H.
\end{align*}
Now taking the inverse of  equation (3.2) we have
\begin{align*}\tag {3.4} 
(k-\lambda)(G-1)=(D^{-1})^2+HD^{-1}.
\end{align*}
Thus from equations (3.3) and (3.4) we must have $D^{-1}H=HD^{-1}$; taking inverses gives
$$DH=HD.$$

Thus $D$ commutes with $G,H,D^{-1}$. 
Now multiplying  $(k-\lambda)(G-1)=D^2+HD$ by $H$ we obtain
$$(k-\lambda)(hG-H)=D\cdot DH+hHD.$$
Multiplying by $H$ again we have
\begin{align*}
\tag {3.5}   h(k-\lambda)(hG-H)=(DH)^2+h^2(DH).\end{align*}

We now find the minimal polynomial for $DH$, by first finding the minimal polynomial for $hG-H$. A calculation shows  that
\begin{align*}
&(hG-H)^2=h^2(h^2-2)G+hH,\\&
(hG-H)^3=h^3(h^4-3h^2+3)G-h^2H.
\end{align*}
Thus $(hG-H), (hG-H)^2, (hG-H)^3$ are  in the span of $H,G$ and so are linearly dependent. 
Define 
$$F(x)=x(x+h)(x-h^3+h).$$ Then one finds that 
\begin{align*}\tag {3.6} 
F(hG-H)=0.
\end{align*}
Now let $\Delta=DH$. Then from (3.5) we have
\begin{align*}\tag {3.7}
hG-H=\frac {1} {h(k-\lambda)} (\Delta^2+h^2\Delta).
\end{align*}
It follows from (3.6), (3.7) that $\Delta$ satisfies the polynomial
$$ x \left( x+{h}^{2} \right)  \left( 2x+{h}^{2}+{h}^{3} \right) 
 \left( 2x+{h}^{2}-{h}^{3} \right)  \left( 2x+{h}^{2} \right) ^{2}.
$$

For $n \in \mathbb N$ we let $1_n=(1,1,\dots,1), 0_n=(0,0,\dots,0)  \in \mathbb R^n$.
Now from equation (3.5) and the definition of the function $F$ we see that: 

\noindent (A) the matrix 
$hG-H$  has eigenvalue $\mu=h^3-h$ with  an eigen space containing   $1_{h^2}$.\medskip 

\noindent (B)  the matrix  $hG-H$  has eigenvalue $\mu=-h$ with corresponding eigenspace containing the span of 
\begin{align*}
(1_h,0_h,0_h,\dots,0_h,-1_h),
(0_h,1_h,0_h,\dots,0_h,-1_h),
\dots
,(0_h,0_h,0_h,\dots,0_h,1_h,-1_h),
\end{align*} so that this eigenspace has dimension at least  $h-1$.\medskip 

\noindent (C) Lastly, $hG-H$  has eigenvalue $\mu=0$ with corresponding eigenspace containing the span of all vectors of the form
$(v_1,v_2,\dots,v_h)$, where $v_i \in \mathbb R^h$ satisfies $J_hv_i=0.$
Thus this eigenspace has dimension at least $h^2-h$. \medskip

Since $1+(h-1)+(h^2-h)=h^2$ 
we see that (A), (B), (C)  describe all the eigenspaces, and we conclude that $hG-H$ is diagonalizable. 

The eigenvalues for $(k-\lambda)h(hG-H)$ are thus
$$\mu'=(k-\lambda)h^2(h^2-1),\quad  \mu'=-h^2(k-\lambda),\quad \mu'=0,$$ with corresponding eigenspaces $E_{\mu'}$, as given in (A), (B), (C). 

\medskip

Let $g_1=1,g_2,\dots,g_h$ be coset representatives for $G/H$.
Let $d_i=|D \cap Hg_i|$, so that $d_1=0$. Let $D=(D_{ij})$, where the blocks are of size $h \times h$ and are $\{0,1\}$ matrices. 

Now from $DH=HD$ we see that $J_hD_{ij}=D_{ij}J_h$ for all $1\le i,j \le h$. 

 
 \begin{lem} \label {lemdiju}
 Let $A$ be an $h \times h$ matrix whose entries are $0,1$, and such that $J_hA=AJ_h$. Then every row and column of $A$ has the same number of $1$s in it.\end{lem}
 \noindent {\it Proof}
Note that the $k$th column of $J_hA$ is $u(1,1,\dots,1)^T$, where $u$ is the number of $1$s in the $k$th column of $A$. Similarly,
the $k$th row of $AJ_h$ is $u(1,1,\dots,1)$, where $u$ is the number of $1$s in the $k$th row of $A$. 
 
 Let $a_i, 1 \le i\le h,$ be the number of $1$s in the $i$th row of $A$.
 Then the $i$th  row of $AJ_h$ is $(a_i,a_i,\dots,a_i)$. Thus 
 $$J_hA=\begin {pmatrix} a_1&a_1&\dots &a_1\\
  a_2&a_2&\dots &a_2\\
 \vdots&\vdots&\ddots&\vdots\\
 a_h&a_h&\dots &a_h
 \end{pmatrix}.$$
 Let $b_i, 1\le i\le h$ be the number of $1$s in the $i$th column of $A$. Then the $i$th column of $AJ_h$ is $(b_i,b_i,\dots,b_i)^T$, so that we have
 $$AJ_h=\begin {pmatrix} b_1&b_2&\dots &b_h\\
 b_1&b_2&\dots &b_h\\
 \vdots&\vdots&\ddots&\vdots\\
b_1&b_2&\dots &b_h
 \end{pmatrix}.$$
 Since $AJ_h=J_hA$ we see from the first columns of these matrices that $$b_1=a_1=a_2=\dots=a_h,$$ and from the first rows that
$$a_1=b_1=b_2=\dots=b_h.$$ But this shows that $a_i=a_j=b_r=b_s$ for all $1\le i,j,r,s \le h$, and the result follows.
 \qed \medskip

Thus from $HD_{ij}=D_{ij}H$ we see that each row and column of $D_{ij}$ has the same number of $1$s in it. Let this number be $d_{ij}$, so that $d_{ii}=0$.
Thus
$DH=HD=(d_{ij}J_h)$. 

Now $DD^{-1}=D^{-1}D$ with $D^{-1}=D^T$ shows that $D$ is a normal matrix. Clearly $H$ is a normal matrix. Thus we have
\begin{lem} \label{lemdiag} The matrices 
$D, H, G $ are commuting normal matrices and  are simultaneously diagonalizable. 
\qed
\end{lem}

In particular $DH$ is diagonalizable. 
Next: if $\alpha $ is an eigenvalue for $\Delta=DH$ with eigenvector $v$, then 
$$(\Delta^2+h^2\Delta)v=(\alpha^2+h^2\alpha)v.$$
But $\Delta^2+h^2\Delta=(k-\lambda)h(hG-H)$, which shows that $v$ is also an eigenvector for $(k-\lambda)h(hG-H)$ with eigenvalue
$\alpha^2+h^2\alpha$. However we know the eigenvalues and eigenvectors for $(k-\lambda)h(hG-H)$. Thus there are three cases:

\noindent (A) Here 
$\alpha^2+h^2\alpha=(k-\lambda)h^2(h^2-1)$, in which case we solve for $\alpha$:
$$\alpha=\frac 1 2 \left (\pm h^3-h^2\right ).$$
Here the eigenvector is $1_{h^2}$. Since $\Delta^2+h^2\Delta$ is a matrix with non-negative entries it follows that
 $\frac 1 2 \left (- h^3-h^2\right )$ is not possible with this eigenvector. Thus we only have $\frac 1 2 \left ( h^3-h^2\right )$ as an eigenvalue  in this case.
\medskip

\noindent (B) Here 
$\alpha^2+h^2\alpha=-(k-\lambda)h^2,
$ so that
$\alpha=-h^2/2$. 
\medskip

\noindent (C) Here 
$\alpha^2+h^2\alpha=0$, so that $\alpha=0,-h^2$. 
  \medskip
  
  Since $DH$ is diagonalizable the dimensions of the eigenspaces in cases (A), (B), (C) must be $1,h-1,h^2-h$ (respectively). In particular, each
  eigenvector for $hG-H$ as in (B) is also an eigenvector for $DH$. Thus we have
  \begin{align*}
  \begin{pmatrix} 0&d_{12}J_h&d_{13}J_h&\dots&d_{1h}J_h\\
  d_{21}J_h&0&d_{23}J_h&\dots &d_{2h}J_h\\
  d_{31}J_h&d_{32}J_h&0&\dots&d_{3h}J_h\\
  \vdots&  \vdots&  \vdots&  \ddots&  \vdots\\
  d_{h1}J_h&d_{h2}J_h&d_{h3}J_h&\dots&0\end{pmatrix}
  \begin{pmatrix} 1_h\\0_h\\0_h\\ \vdots \\
  -1_h\end {pmatrix} =-\frac {h^2} 2   \begin{pmatrix} 1_h\\0_h\\0_h\\ \vdots \\
  -1_h\end {pmatrix},
  \end{align*} which, since $J_h 1_h=h 1_h$, gives 
  $$d_{1h}=\frac h 2, d_{21}=d_{2h},d_{31}=d_{3h},\dots,d_{h-1,1}=d_{h-1,h},d_{h1}=\frac h 2.$$
  Similarly, using
   \begin{align*}
  \begin{pmatrix} 0&d_{12}J_h&d_{13}J_h&\dots&d_{1h}J_h\\
  d_{21}J_h&0&d_{23}J_h&\dots &d_{2h}J_h\\
  d_{31}J_h&d_{32}J_h&0&\dots&d_{3h}J_h\\
  \vdots&  \vdots&  \vdots&  \ddots&  \vdots\\
  d_{h1}J_h&d_{h2}J_h&d_{h3}J_h&\dots&0\end{pmatrix}
  \begin{pmatrix} 0_h\\1_h\\0_h\\ \vdots \\
  -1_h\end {pmatrix} =-\frac {h^2} 2   \begin{pmatrix} 0_h\\1_h\\0_h\\ \vdots \\
  -1_h\end {pmatrix},
  \end{align*} we obtain
 \begin{align*}d_{12}&=d_{1h}=\frac h 2,d_{2h}=\frac h 2=d_{21},d_{32}=d_{3h}=d_{31},\dots,\\&d_{h-1,2}=d_{h-1,h}=d_{h-1,1},d_{h2}=\frac h 2.
 \end{align*}
  Continuing we see that 
  $$d_{ij}=\frac h 2 \text {  for all } 1\le i\ne j\le h.$$ 
  
This shows that $|D \cap gH|=\frac h 2$ for all $g \notin H$, and so also gives
 \begin{align*}\tag {3.8} DH=HD=\frac h 2 (G-H). \end{align*}
  
  \begin{prop} \label{lemnorm}
  Let $H \le G, |H|=h, |G|=n,$ and order the elements of $G$ according to the cosets of $H$ as in the above. Represent elements of $G$ using the regular representation relative to this ordering. 
   Then for $g \in G$ we write $g=(g_{ij})$, where each $g_{ij}$ is a $0,1$ matrix of size  $h \times h$. Then $H\triangleleft G$ if and only if for all $g \in G$ and all $1\le i,j\le n/h$ each $g_{ij}$ is either the zero matrix or  is a permutation matrix.
  \end{prop} 
  \noindent {\it Proof}  We note that 
   $H\triangleleft G$ if and only if for all $g \in G$ we have $Hg=gH$, where $H=\mathrm{diag}(J_h,J_h,\dots,J_h)$. 
   
   Assume that $H\triangleleft G, g \in G, g=(g_{ij})_{1\le i,j \le h}$, where each $g_{ij}$ is an $h \times h$ matrix. Then $gH=Hg$ implies that $g_{ij}J_h=J_hg_{ij}$ for all $1\le i,j \le n/h$. 
   By Lemma \ref {lemdiju} this is true if and only if all the rows and columns of $g_{ij}$ have the same number of $1$s in them. Since each row and column of $g$ has exactly one $1$ in it (the rest of the entries being $0$) we see that if $g_{ij}\ne 0$, then each row and column of $g_{ij}$ has exactly one $1$ in it. Thus, for fixed $i,j$,  no other $g_{ik}, k \ne j,$ or $g_{kj}, k \ne i,$ can be non-zero. In particular, each $g_{ij}$ is a permutation matrix.
   
   Now assume that for all $g \in G$ and all $1\le i,j\le n/h$ each $g_{ij}$ is either the zero matrix or  is a permutation matrix. We wish to show that 
    $H\triangleleft G$ i.e. that $g_{ij}J_h=J_hg_{ij}$ for all $1\le i,j \le n/h. $  This is certainly true if $g_{ij}=0$, while if $g_{ij}$ is a permutation matrix, then $g_{ij}J_h=J_h=J_hg_{ij}$, and so we are done.\qed\medskip

Let $D$ denote a  difference set where $G=D \cup D^{-1}\cup H, H \le G, D \cap H=D \cap D^{-1}=\emptyset$. 
   We now wish to show that  $H\triangleleft G$.
   
   From the above we know that $|G|=h^2, h=|H|$, where $h$ is even and $D=(d_{ij})$, where either $D_{ij}=0$ or $D_{ij}$ is a $0,1$ matrix that  has $h/2$ $1$s in each row and column. We wish to show that $gH=Hg$ for all $g \in G$. This is certainly true if $g \in H$, so assume that $g \notin H$. 
   Write $g=(g_{ij})$ as in the above. Since $g \notin H$ we either have $g \in D$ or $g \in D^{-1}$. Without loss of generality we can assume that $g \in D$.
   Now either $D_{ij}=0$ or $D_{ij}$ is a $0,1$ matrix that  has $h/2$ $1$s in each row and column, so either $g_{ij}=0$ or $g_{ij}$ is a $0,1$ matrix that  has one  $1$ in each row and column. It follows that $g_{ij}J_h=J_hg_{ij}$, and so $H\triangleleft G$.\qed\medskip 
   
   We have thus proved the following, which includes proofs of Theorem \ref {thm1} (ii), (iii):
   
   \begin{thm} \label {thm111}
Let $D$ denote a  difference set where $G=D \cup D^{-1}\cup H, H \le G, D \cap H=D \cap D^{-1}=\emptyset$. Then

\noindent (i) $|G|=h^2$ where $h=|H|$ and $h$ is even;

\noindent (ii) $H\triangleleft G$;

\noindent (iii) each  coset  $Hg\ne H$ meets $D$ in $h/2$ points;

\noindent (iv) if we write $D=(D_{ij})$, with $h \times h$ blocks $D_{ij}$, then  each $D_{ij}$ is either a zero matrix or is a $0,1$ matrix with $h/2$   $1$'s in each row and column. 
\qed 
\end{thm}
   
   For Theorem \ref {thm1} (iv) we note that if $g \in G$ is an involution that is not in $H$, then we must have $g\in D$ or $g \in D^{-1}$. In either case we have $g \in D \cap D^{-1}$, a contradiction.

     For Theorem \ref {thm1} (v) we show that $H\triangleleft G$ does not have a complement. So suppose that $L \le G$ is 
a complement to $H$. 
Now since $L$ is a complement to $H$ we have $|L|=|G|/|H|=h$, which is even. Thus $L$ contains an involution that is not in $H$, a contradiction.

   This now concludes the proof of Theorem \ref {thm1}. \qed
   
     \section {    $H$ and subgroups of index $2$}
     
     We prove the following result, which gives Theorem \ref {thmfr} (i):
  
  \begin{thm} \label {tmfr}  Let $G$ be a $(v,k,\lambda)_0$ relative skew Hadamard difference set group with  difference set $ D $ and subgroup $H$.
  
  Then the  subgroup $H $ is contained in every index $2$ subgroup of $G$.
  Further, $D$ meets each such  subgroup in exactly $\lambda$ points.
  
If  $G$ is a $2$-group,  
  then $H$ is a subgroup of the Frattini subgroup of $G$.

  \end{thm}
  \noindent {\it Proof}  From  Lemma \ref {lem1}  we know that  $|G|=h^2, k=\frac {h(h-1)}{2},\lambda=\frac {h(h-2)}{4}$.
  Let $M\le G$ be a   subgroup of index $2$ and let $\pi:G \to G/M=\langle t:t^2=1\rangle$ be the quotient map.
  Let 
  $$|D \cap M|=d_1,\quad |H \cap M|=h_1,$$ so that
  $$\pi(D)=d_1\cdot 1+(k-d_1)t,\quad  \pi(H)=h_1\cdot 1 +(h-h_1) t.$$
  Let $d_2=k-d_1, h_2=h-h_1.$ Then we have the equations
  \begin{align*}\tag {4.1} &
  d_1+d_2=k,\quad h_1+h_2=h,  \quad  k=h(h-1)/2,\quad \lambda=h(h-2)/4.
  \end{align*}
  Now from equations (3.2) and  (3.8) we deduce that
  $$D^2=\lambda G+\frac h 2H-(k-\lambda)1.$$ 
  Taking the image of this under $\pi$, and using the fact that $\pi(D)=d_11+d_2t$, we obtain two equations (by looking at the coefficients of $1$ and $t$):
  \begin{align*}&
  \tag {4.2}d_1^2+d_2^2=\lambda h^2/2+hh_1/2+\lambda-k;\\
  &2d_1d_2=\lambda h^2/2+hh_2/2.
  \end{align*}
  Now $D+D^{-1}=G-H$   gives (by acting by $\pi$)
   \begin{align*}&
2d_1+2d_2t=h^2/2(1+t)-(h_1+h_2t),
\end{align*}
which gives
 \begin{align*}&
  \tag {4.3}2d_1=h^2/2-h_1,\quad 2d_2=h^2/2-h_2.\end{align*}
  Solving   equations (4.1), (4.2), (4.3)  we find that 
  $$h_1=h,\quad h_2=0,\quad d_1=\lambda,\quad d_2=k-\lambda.$$
  Thus $D$ meets $M$ in $\lambda$ points, and all of $H$ is in $M$. Since this is true for any maximal subgroup $M$ we see that $H$ is contained in the Frattini subgroup if $G$ is a $2$-group, since every maximal subgroup of such a group has index  $2$.
  \qed

    \section {Proof of the $p$ odd case}
    
    Let $N\triangleleft G$ be of index $p$, a prime. Let $$\pi:G \to Q=G/N=\langle t\rangle\cong \mathcal C_p,$$ be the quotient map.
Then
we have 
$$\pi\left (\sum_{g \in G} g\right )=\frac {h^2}{p} \sum_{q \in Q} q=\frac {h^2}{p} \sum_{i=0}^{p-1} t^i.$$
We let
$$\pi(D)=\sum_{i=0}^{p-1}x_it^i,\quad \pi(H)=\sum_{i=0}^{p-1}y_it^i,$$ where $x_i,y_i \ge0$ are integers, and
$$\sum_{i=0}^{p-1} x_i=k,\quad \sum_{i=0}^{p-1}y_i=h.$$

We may represent elements of $Q$ as matrices where the generator $t$ corresponds to 
$P=\begin{pmatrix} 0&1&0&0&\dots \\
 0&0&1&0&\dots\\ 
 \vdots&\vdots&\vdots&\ddots&\vdots \\
  0&0&0&\dots&1 \\
  1&0&0&\dots&0
  \end{pmatrix}$.
Now we can simultaneously  diagonalize  $\pi(D),\pi(D^{-1}),$ $ \pi(G),$ $ \pi(H)$ using the matrix
$R=(\zeta^{-(i-1)(j-1)})$, where $\zeta=\exp{2\pi i/p}$.
We first note from \cite [\S 3.2] {da} that
$$RPR^{-1}=\mathrm{diag}(1,\zeta,\zeta^2,\dots,\zeta^{p-1}).$$ 
From this it follows that

\begin{align*}&
  \tilde H=RHR^{-1}=\\& \begin{pmatrix} 
  \sum_{i=0}^{p-1}y_i&0&0&0&\dots&0 \\
 0&\sum_{i=0}^{p-1}y_i\zeta^i&0&0&\dots&0\\ 
 0&0&\sum_{i=0}^{p-1}y_i\zeta^{2i}&0&\dots &0\\
 \vdots&\vdots&\vdots&\ddots&\vdots&\vdots \\
  0&0&0&\dots&\sum_{i=0}^{p-1}y_i\zeta^{(p-2)i}&0 \\
  0&0&0&\dots&0&\sum_{i=0}^{p-1}y_i\zeta^{(p-1)i}
  \end{pmatrix};\\
  &\tag {5.1} \label{t50} \tilde D=RDR^{-1}=\\& \begin{pmatrix} \sum_{i=0}^{p-1}x_i&0&0&0&\dots \\
 0&   \sum_{i=0}^{p-1}x_i\zeta^i  &0&0&\dots\\
 0&0&\sum_{i=0}^{p-1}x_i\zeta^{2i}&0&\vdots\\ 
 \vdots&\vdots&\vdots&\ddots&\vdots \\
  0&0&\dots&\sum_{i=0}^{p-1}x_i\zeta^{(p-2)i}&0 \\
  0&0&\dots&0&\sum_{i=0}^{p-1}x_i\zeta^{(p-1)i}
  \end{pmatrix};\\
    &\tilde D^{-1}=RD^{-1}R^{-1}=\\&\quad  \begin{pmatrix} \sum_{i=0}^{p-1}x_i&0&0&0&\dots \\
 0&   \sum_{i=0}^{p-1}x_i\zeta^{-i}  &0&0&\dots\\
 0&0&\sum_{i=0}^{p-1}x_i\zeta^{-2i}&0&\vdots\\ 
 \vdots&\vdots&\ddots&\vdots&\vdots \\
  0&0&\dots&\sum_{i=0}^{p-1}x_i\zeta^{-(p-2)i}&0 \\
  0&0&\dots&0&\sum_{i=0}^{p-1}x_i\zeta^{-(p-1)i}
  \end{pmatrix}.
  \end{align*} We also have 
  \begin{align*}
  &\tilde G=RGR^{-1}= \begin{pmatrix} h^2&0&0&\dots&0 \\
 0&0&0&\dots&0\\ 
 \vdots&\vdots&\vdots&\ddots&\vdots  \\
  0&0&0&\dots&0 \\
  0&0&0&\dots&0
  \end{pmatrix}.
    \end{align*}

From $H^2=hH$ we see that the minimal polynomial of $H$ is $x(x-h)$, and so the  minimal polynomial of $\tilde H$ is a divisor of $x(x-h)$. In particular, the eigenvalues of $\tilde H$ are either $0$ or $h$.
Now we know that $\sum_{i=0}^{p-1}y_i=h$, and for $1\le j\le p-1$ we must also have
\begin{align*}\tag {5.2} \label {t51}
\sum_{i=0}^{p-1}y_i\zeta^{ij} \in \{0,h\}.
\end{align*}

We rewrite equations (\ref {t51})  in matrix form as
\begin{align*}\tag {5.3}\label {t52}
\begin{pmatrix} 1&1&1&1&\dots&1\\
1&\zeta &\zeta^2&\zeta^3&\dots&\zeta^{p-1}\\
1&\zeta^{2}&\zeta^{4}&\zeta^6&\dots&\zeta^{2(p-1)}\\
\vdots&\vdots&\vdots&\dots&\vdots\\
1&\zeta^{p-1}&\zeta^{2(p-1)}&\zeta^{3(p-1)}&\dots&\zeta^{(p-1)^2}
\end{pmatrix} 
\begin{pmatrix} y_0\\y_1\\y_2\\y_3\\
\vdots\\y_{p-1}
\end{pmatrix} =
\begin{pmatrix} h\\
0 \text { or } h\\
0 \text { or } h\\
0 \text { or } h\\
\vdots \\
0 \text { or } h
\end{pmatrix}.
\end{align*}
Let $T$ denote the matrix
in (\ref{t52}). Then $T^{-1}=\frac 1 p T^*$.
Thus
\begin{align*}
\begin{pmatrix} y_0\\y_1\\y_2\\y_3\\
\vdots\\y_{p-1}
\end{pmatrix}=\frac 1 p \begin{pmatrix} 1&1&1&1&\dots&1\\
1&\zeta^{-1} &\zeta^{-2}&\zeta^{-3}&\dots&\zeta^{-(p-1)}\\
1&\zeta^{-2}&\zeta^{-4}&\zeta^{-6}&\dots&\zeta^{-2(p-1)}\\
\vdots&\vdots&\vdots&\dots&\vdots\\
1&\zeta^{-(p-1)}&\zeta^{-2(p-1)}&\zeta^{-3(p-1)}&\dots&\zeta^{-(p-1)^2}
\end{pmatrix} 
\begin{pmatrix} h\\
0 \text { or } h\\
0 \text { or } h\\
0 \text { or } h\\
\vdots \\
0 \text { or } h
\end{pmatrix}.
\end{align*}

Now, since $p$ is a prime,  the minimal polynomial for $\zeta$ is $\sum_{i=0}^{p-1}x^i$. We also have  $y_0,\dots,y_{p-1} \in \mathbb Z$.
Thus each choice of $0$ or $h$ in the above equation must be the same. Call this choice $\eta$.
Thus we have two possibilities:

\noindent (a)  $\eta=h$ so that $y_0=h,y_1=y_2=\dots=y_{p-1}=0$, or 

\noindent (b)  $\eta=0$, so that 
 $y_0=y_1=y_2=\dots=y_{p-1}=\frac h p$.

If we have (a), then    $\tilde H=hI_p$, so that the equation
$\tilde D \tilde H=\frac h 2 (\tilde G-\tilde H)$ gives 
\begin{align*}
\tilde D=\frac 1 2 (\tilde G-\tilde H)=\frac 1 2 \begin{pmatrix} h^2-h&0&0&\dots &0\\
0&-h&0&\dots&0\\
0&0&-h&\dots &0\\
\vdots&\vdots&\vdots&\ddots&\vdots\\
0&0&0&\dots&-h
\end {pmatrix}.
\end{align*}
Using the expression for $\tilde D$ given in equation (\ref {t50}) we see that the above is expressible as a matrix equation:
\begin{align*}
\tag {5.4}\label {t53}
\begin{pmatrix} 1&1&1&1&\dots&1\\
1&\zeta &\zeta^2&\zeta^3&\dots&\zeta^{p-1}\\
1&\zeta^{2}&\zeta^{4}&\zeta^6&\dots&\zeta^{2(p-1)}\\
\vdots&\vdots&\vdots&\dots&\vdots\\
1&\zeta^{p-1}&\zeta^{2(p-1)}&\zeta^{3(p-1)}&\dots&\zeta^{(p-1)^2}
\end{pmatrix} 
\begin{pmatrix} x_0\\x_1\\x_2\\x_3\\
\vdots\\x_{p-1}
\end{pmatrix} =\frac 1 2 
\begin{pmatrix} h^2-h\\
-h\\
 -h\\
-h\\
\vdots \\
 -h
\end{pmatrix}.
\end{align*}
As before, if we denote the matrix in (\ref {t53}) by $U$, then $U^{-1}=\frac 1 p U^*$, and we obtain
$$x_0=\frac {1}{2p} h(h-p);\,\,\, x_1=x_2=\dots=x_{p-1}=\frac {1} {2p} h^2.
$$
     This concludes consideration of (a). 
    \medskip

    We now show that (b) is not possible.  So assume (b). Then as noted we get  $y_0=y_1=y_2=\dots=y_{p-1}=h/p.$
Thus $ H=\frac h p J_p$, so that $\tilde H=\begin{pmatrix} h&0&0&\dots&0\\
0&0&0&\dots&0\\ \vdots&\vdots&\vdots&\dots&\vdots\\
0&0&0&\dots&0\end{pmatrix}.$

Now from $\tilde D \tilde D^{-1}=\lambda \tilde G +(k-\lambda)$ we get 
\begin{align*}&\tag {5.5}\label{t54}
 \begin{pmatrix}\left  (\sum_{i=0}^{p-1} x_i\right )^2&0&0&\dots\\
 0&(\sum_{i=0}^{p-1} x_i\zeta^i)(\sum_{i=0}^{p-1} x_i\zeta^{-i})&0&\dots\\
 \vdots&\vdots&\vdots&\ddots
 \end{pmatrix}
 \\&=\lambda   \begin{pmatrix} h^2&0&0&\dots\\
 0&0&0&\dots\\
  \vdots&\vdots&\vdots&\ddots
  \end{pmatrix}+
  (k-\lambda) \begin{pmatrix} 1&0&0&\dots\\
 0&1&0&\dots\\
  \vdots&\vdots&\vdots&\ddots
  \end{pmatrix}.
  \end{align*}

Now the $x_i \in \mathbb Z$ and so from the (2,2) entry of (\ref{t54}) we see that (taking the rational part) we have 
\begin{align*}\tag {5.6}\label{t56}
\sum_{i=0}^{p-1}x_i^2=k-\lambda.\end{align*}
Using $\sum_{i=0}^{p-1}x_i=k$, the fact that $x_i \in \mathbb Z,$  and (\ref{t56}) gives 

$$k-\lambda= \sum_{i=0}^{p-1}x_i^2\ge \sum_{i=0}^{p-1}x_i=k,$$ so that  
$\lambda \le 0,$ a contradiction. This proves Theorem \ref {thmfr} (a) (i) and (ii). \qed\medskip 

Since we only have (a) we see that $|N \cap D|=\frac 1 {2p} h(h-p)$, so that in the case $p=2$ we have $|N \cap D|=\frac 1 {4} h(h-4)=\lambda,$ which gives 
Theorem \ref {thmfr}  (b).\qed 

    
\section{The Schur ring and minimal polynomials}

We have $(G-H)^{-1}=G-H, (H-1)^{-1}=H-1,(D^{-1})^{-1}=D$, and so we just need to show that $D,D^{-1},H-1,1$ commute and span the ring that they generate. We have already seen in Lemma \ref {lemdiag}  that they commute.

We have 
$H\cdot G=hG,  D\cdot G=kG=D^{-1}\cdot G.$
 
  Using  equations (3.2) and (3.8)
   we get 
   $$D^2=(k-\lambda)(G-1)-\frac h 2 (G-H).$$

 We collect together the rest of the  products that we need:
\begin{align*}
&
HD=DH=\frac h 2 (G-H);\\& H^2=hH,\\& 
D^2=(k-\lambda)(G-1)-\frac h 2 (G-H)
=(k-\lambda-\frac h 2 )(D+D^{-1})+(k-\lambda) (H-1),\\&
D\cdot D^{-1}=D^{-1}\cdot D=\lambda G+(k-\lambda)1=\lambda D+\lambda D^{-1}+\lambda (H-1)+k1.
\end{align*} 
  Since $k=h(h-1)/2,\lambda=h(h-2)/4, k-\lambda=h^2/4\in \mathbb Z,$ one can check that all the coefficients in the above sums are non-negative integers.
  This proves that $D,D^{-1},H-1,1$ commute and span the ring that they generate. Theorem \ref {thmsring} follows. \qed \medskip

   One can show that
  \begin{align*}&D^3=\frac {h^2}{4}D^{-1}+\left (\frac 1 8 h^4-\frac 3 8 h^3+\frac 1 4 h^2\right )G;\\&
   D^4=\left (\frac 1 {16}h^6-\frac 1 4 h^5+\frac 3 8 h^4-\frac 1 4 h^3\right )G+\frac 1 {16} h^4.
   \end{align*}

  For a matrix or an element $M$ of an algebra we let $\mu(M)$ denote the minimal polynomial of $M$.
  
    To help us find   $\mu(D)$   
  we have the equations 
  \begin{align*}
  &  G=D+D^{-1}+H,\quad DD^{-1}=\lambda G+(k-\lambda),\quad
  DH=\frac h 2 (G-H),\\& 	D^{-1}H=\frac h 2 (G-H),\quad D^2=(k-\lambda)(G-1)-\frac h 2 (G-H).\end{align*}
  Thinking of $D,D^{-1},G,H$ as variables in a polynomial ring over $\mathbb Q(h)$ 
    we can consider the ideal generated by the above elements. This has dimension zero, and calculating a Gr\" obner basis for this ideal  and doing an elimination,   one finds that $D$ satisfies the   polynomial  $(x-k)\left (x+\frac h 2\right  )\left (x^2+\frac {h^2} 4\right )$. Thus $\mu(D)$ divides this polynomial. 
    
    We note that $\frac 1 k D$ is a stochastic matrix, and since  $D^2=(k-\lambda)(G-1)-\frac h 2 (G-H)$ it follows that 
    \begin{lem} \label{lemdstoc} The matrix 
    $\frac 1 k D$ is an  irreducible doubly stochastic matrix.\qed\end{lem}
 
 Further, we know that $\mu(D)$  factors as a product of distinct linear factors $(x-\kappa)$, where $\kappa$ is an eigenvalue (since $D$ is diagonalizable by Lemma \ref {lemdiag}).

 Next we note that $k$ is an eigenvalue of $D$, since each row sum and column sum  of $D$ is $k$.
 Next we show that $-h/2$ is an eigenvalue of $D$: for $g \notin H$ we have $H-Hg \ne 0$ and 
 \begin{align*}  
D\cdot(  H-Hg)&=DH(1-g)=\frac h 2 (G-H)(1-g)\\
&=\frac h 2 (G-H-G+Hg)=-\frac h 2 (H-Hg).
\end{align*}
Thus $-\frac h 2$ is an eigenvalue for $D$.

Since $D$ is a matrix with real entries it follows that  the eigenspaces for eigenvalues $\pm ih/2$ have  the same dimension, and that
either $\mu(D)=(x-k)(x+h/2)$ or $\mu(D)=(x-k)(x+h/2)(x^2+\frac {h^2}{4}).$ 
If $\mu(D)=(x-k)(x+h/2)$, then, since $D$ is diagonalizable, Lemma \ref {lemdstoc} and the Perron Frobenius theorem show that $D$ has eigenvalue $k$ with multiplicity one, and 
$-h/2$ with multiplicity $h^2-1$. Now, since $D \cap H =\emptyset,$ we see that  $D$ has trace zero. Thus we must have
$$k+(h^2-1)(-h/2)=0,$$ but the lefthand side of this expression is $-h^2(h-1)$, which gives a contradiction. Thus  $\mu(D)=(x-k)(x+h/2)(x^2+\frac {h^2}{4}).$ In fact it easily follows from Trace$(D)=0$  that the eigenvalue $-h/2$ has multiplicity $h-1$. \qed\medskip 

This  gives a proof of Theorem \ref {thmminpoly}.\qed  \medskip

  Using the above ideal it is easy to show that $G \setminus D$ is another difference set, where
  $$(G-D)(G-D)^{-1}=(G-D)(G-D^{-1})=\left (\frac 1 4 h^2+\frac 1 2 h\right ) G+\frac 1 4 h^2.$$
  
  Further, we have
  \begin{align*}
  (2D-G)(2G-D)^T&=(2D-G)(2G-D^{-1})\\&=4DD^{-1}-2DG-2GD^{-1}-G^2\\&=4\lambda G-2kG-2kG+h^2G=h^2,
  \end{align*}
  as required for a part of
  Theorem \ref{thmhad}. As we have seen $D,G$ are simultaneously diagonalizable, and
  $2D-G$ is similar to
  \begin{align*}
  &\mathrm{diag}(2k,-h,-h,\dots,ih,\dots,ih,-ih,\dots,-ih)-\mathrm{diag}(h^2,0,\dots,0)\\
  &=\mathrm{diag}(-h,-h,\dots,ih/2,\dots,ih/2,-ih/2,\dots,-ih/2),
  \end{align*}
  and it follows that $\mu(2D-G)=(x+h)(x^2+h^2)$. This gives   Theorem \ref{thmhad}.\qed

   \section {Examples of relative skew Hadamard difference set groups }
   
  The groups $\mathfrak G_{n,k}$ have been defined in the introduction. We now show that they are relative skew Hadamard difference set groups.
   We let 
   $$H=\langle b_1,b_2,\dots,b_n\rangle.$$
Then a transversal for $H$ in $G$ is the set of products $a_{i_1}a_{i_2}\cdots a_{i_u}$, where these are indexed by the sequences $i_1<i_2<\cdots<i_u$ of $1,2,\dots,n$, or in other words, indexed by the subsets $X=\{i_1,i_1,\dots,i_u\}$ of $\{1,2,\dots,n\}$. We let $a_X=a_{i_1}a_{i_2}\cdots a_{i_u}$ denote the corresponding element of $G$. Here $a_\emptyset=1$.
   We may also employ a similar notation for the elements $b_X=b_{i_1}b_{i_2}\cdots b_{i_u}$.

We note that for any $g \in G$ we have $g^2 \in H$.
We are interested in  the hypothesis
\medskip 

\noindent (H1): there is a set of distinct maximal subgroups $M_1,\dots,M_{2^n-1}$  of $H$, and  an ordering $S_1,\dots,S_{2^n-1}$ of the non-empty subsets of $\{1,\dots,n\}$ so that $a_{S_i}^2 \notin M_i$.
\medskip

\begin{prop} \label{lemgnkh1} The groups  $\mathfrak G_{n,k}$ satisfy (H1).
\end{prop}
\noindent {\it Proof}
 We  first show that the squares of the coset representatives $a_S, S \subseteq \{1,2,\dots,n\}$, are distinct.
 
 We   note that the subgroup $J=\langle a_2,a_3,\dots,a_n\rangle$ is isomorphic to $\mathcal C_4^{n-1}$. We also have $J \triangleleft \mathfrak G_{n,k}$, so that
 $\mathfrak G_{n,k}=J\rtimes \langle a_1\rangle=J \rtimes \mathcal C_4$. 
 
 If $S \subseteq \{1,2,\dots,n\}$ and $ m \in \mathbb Z$ we let $S+m$ be the set $\{u+m:u \in S\}$, where we take numbers mod $n$ so that $S+m \subseteq \{1,2,\dots,n\}$.
 
 Now for a coset representative $a_S, S=\{i_1,i_2,\dots,i_u\} \subseteq \{2,\dots,n\},$ we have $a_S \in J$ and so
from the relations in $\mathfrak G_{n,k}$ we have 
 $$a_S^2=b_{i_1+k}b_{i_2+k}\dots b_{i_u+k}=b_{S+k}.$$
 We note that in this situation, since $1 \notin S$, we have $1+k \notin S+k$. 

Now for a coset representative $a_S$ that is not in $J$ we can write $S=\{1,i_1,i_2,\dots,i_u\}$, where $a_{S \setminus \{1\}}\in J$.  
  So if we let $K=S \setminus \{1\}$, then $a_S=a_1a_K$.
  
  Now write $K=K_1\cup K_2$, where the elements  $a_m,m \in K_2,$ commute with $a_1$, and those  $a_m,m \in K_1,$ do not. 
  Note that   
  $$K_1\subseteq \{2,\dots,k+1\}, \,\,\,\, K_1 \cap K_2=\emptyset,\,\,\,\,  S=\{1\} \cup K_1 \cup K_2.$$
  
  
  Then from the relations in $\mathfrak G_{n,k}$ we have:
  $$a_{K_2}^{a_1}=a_{K_2},\quad a_{K_1}^{a_1}=a_{K_1}b_{K_1-1}.$$
  Thus we have
  \begin{align*}a_S^2=
  (a_1a_{K_1}a_{K_2})^2&=a_1^2a_{K_1}^{a_1}a_{K_1}a_{K_2}^2\\
  &=b_{1+k}\cdot a_{K_1}b_{K_1-1}\cdot a_{K_1}\cdot a_{K_2}^2\\
  &\tag {7.1}\label {t71} =b_{1+k}b_{K_1-1}b_{K_1+k}b_{K_2+k}=b_{K_1-1}b_{S+k}.
  \end{align*}
  
  We next show that $b_{1+k}$ has non-zero exponent in     (\ref{t71}). But from the above we know that $K_1 \subseteq \{2,3,\dots,k+1\},$ so that $1+k \notin K_1-1.$
  If $1+k \in K_i+k, i=1,2,$ then $1 \in K_i,$ a contradiction.  This shows that $b_{1+k}$
  has non-zero exponent in    (\ref{t71}).
  
  Note that in the above we have also shown (i) of  
  
  \begin{lem} \label{lemfff}   With the above definitions we have:
  
  \noindent (i) the element  $b_{1+k}$ occurs with non-zero coefficient in $a_S^2$ if and only if $1 \in S$. 
  
\noindent   (ii) The squares of the coset representatives $a_S, S \subseteq \{1,2,\dots,n\}$, where $1 \in S,$ are distinct.
  \end{lem}
  \noindent {\it Proof} (ii) We need to show that the map $S \mapsto b_{K_1-1}b_{S+k}$ is injective.
  
  We represent $S$ as a (column) vector $v_S \in V=\mathbb F_2^n$, where the $i$th coordinate of $v_S$ is $1$ if and only if $i \in S$. Then the action on $V$ of replacing $S$ by $S+1$ is determined by the $n \times n$  permutation matrix
  $$P=\begin{pmatrix} 0&0&\dots&0&1\\
  1&0&\dots&0&0\\
  0&1&\dots&0&0\\
  \vdots&\vdots&\ddots&\vdots&\vdots\\
  0&0&\dots&1&0
  \end{pmatrix}.
  $$
  Thus for any $i \in \mathbb Z$ we have 
  $$v_{S+i}=P^iv_S.$$
  Let $0_{k,m}$ denote the $k  \times m$ zero matrix, and let $0_k=0_{k,k}$. If $k\le 0$ or $m \le 0$, then $0_{k,m}$ will denote the empty matrix. 
  Then, the map $S \mapsto K_1$, is determined by the $n \times n$ matrix 
  $$A=\mathrm{diag}(0_{1},I_k,0_{n-k-1}),$$ so that $$v_{K_1}=  Av_S.$$
   Thus the map $S \mapsto b_{K_1-1}b_{S+k}$ is represented by the matrix 
  $P^{-1}A+P^k,$ and  we will be done if we can show that $P^{-1}A+P^k$ is a non-singular matrix in GL$(2,\mathbb F_2)$.  But this is the same as showing that $A+P^{k+1}$ is non-singular, where
   \begin{align*}  &A+P^{k+1}=\\& \tag {7.2} \label{t72} \begin{pmatrix}
  0_1&0_{1,k}&0_{1,n-k-1}\\
  0_{k,1}&I_k&0_{k,n-k-1}\\
  0_{n-k-1,1}&0_{n-k-1,k}&0_{n-k-1}
  \end{pmatrix}+
  \begin{pmatrix} 0_{1,n-k-1}&1&0_{1,k}\\
  0_{k,n-k-1}&0_{k,1}&I_k\\
  I_{n-k-1}&0_{n-k-1,1}&0_{n-k-1,k}                   \end{pmatrix}.
    \end {align*}
  We note that since $k<n-1$ the second matrix is not a diagonal matrix, and that the submatrix  $I_k$ in the second matrix of   equation (\ref {t72})  occurs to the right of the diagonal. (This shows that $A+P^{k+1}$ is singular when $k=n-1$.)
  Thus  the $I_k$ in the second matrix of     equation (\ref {t72}) can be used to column-reduce the $I_k$ in the first matrix to zero. This shows that 
  $A+P^{k+1}$ column-reduces to $P^{k+1}$, which is non-singular, and we are done.\qed\medskip 
  
  Let $V^\times=\mathbb F_2^n \setminus \{0\}$. Then non-empty subsets of $S$ correspond bijectively to  elements of $V^\times$, as explained above.
  Further, maximal subgroups of $H$ correspond to subspaces of $V$ of dimension $n-1$, which, in turn,  are determined by elements of $V^\times$: a vector $v \in V^\times $ determines the subspace $M_v=\{u \in V|u\cdot v=0\}$, where $\cdot $ is the usual dot-product on $V$ taking values in $\mathbb F_2$.  Since $V$ is a vector space over $\mathbb F_2$ the correspondence $v \leftrightarrow M_v$ is bijective. Further, given a maximal subgroup (or subspace) $M$ we let $v_M$ denote the corresponding vector.
  
  Thus the correspondence  of subsets with maximal subgroups that we require is $S \leftrightarrow M_S$ where $v_S \leftrightarrow v_{M_S}$, with 
  $v_S \notin M_S$ i.e.  $v_S \cdot v_{M_S}=1$. 
  But this correspondence  determines  a function
 $$\mu:V^\times \to V^\times, \text { where } v_u\cdot v_{\mu(u)}=1 \text { for all } u \in V^\times.$$ 
 Conversely, such a function determines the correspondence that we want.
  We now show how to construct such a function:
  
   \begin{lem} \label{lemmu}   
  For all $n\in \mathbb N, V=\mathbb F_2^n,$  there is a function $\mu:V^\times \to V^\times$ such that 
  $u\cdot {\mu(u)}=1 $ for all $ u \in V^\times.$
  \end{lem}
  \noindent {\it Proof}   We will show that there is a function $\mu$ that is an involution i.e. where we have $\mu(\mu(v))=v$ for all $v \in V^\times$. 
  For $0 \le k\le n$ we let $$(\underline {1}_k,0)=(1,1,1,\dots,1,0,\dots,0) \in V^\times,$$ where there are $k$ $1$s (so for $k=0$ we have the zero vector of $V$). 
  
  Write $v \in V^\times $ as  $v=(v_1,v_2,\dots,v_n), v_i \in \mathbb F_2$. 
  If $1\le k\le n$ where $v_k=1$ and $v_m=0$ for $k+1\le m\le n$, then we let 
  $$\mu(v)=(\underline {1}_{k-1},0)-v,$$ This  satisfies $\mu(v)\cdot v=1$, as required. Further, since the same $k$ works for $\mu(v)$, we have 
  $$\mu(\mu(v))=
  (1_{k-1},0)-((1_{k-1},0)-v)=v.$$
This defines a function $\mu$ that is an involution.\qed\medskip
  
Lemma \ref {lemmu}  determines the pairing for hypothesis (H1)  for the groups $\mathfrak G_{n,k}$, and concludes the proof of Proposition \ref {lemgnkh1}.\qed\medskip

We will next show 
\begin{prop} \label{lll} 
 The groups $\mathfrak G_{n,k}$ are relative skew Hadamard difference set groups.
\end{prop}
\noindent {\it Proof} We first note that since $b_1,\dots ,b_n$ are central idempotents, all the maximal subgroups of $H$ are normal subgroups of $G$.

As usual, subsets $S$ of $G$ will correspond to elements $\sum_{s \in S} s$, of the group algebra.
We define $D$ as follows:

$$D=\sum_{i=1}^{2^n-1} a_{S_i}M_i.$$

Let $a_i=a_{S_i}$. 
We first show that $(a_iM_i)^{-1}=a_i(H-M_i).$ But this is true if and only if 
$a_i^{-1}M_i=a_i(H-M_i)$ if and only if $M_i=a_i^2(H-M_i)$ if and only if $M_i=H-a_i^2M_i$. But this latter equation is true since $a_i^2 \in H$ and $a_i^2 \notin M_i$.

Thus we have: 
$$D^{-1}=\sum_{i=1}^{2^n-1}a_{S_i}(H-M_i).$$

Let $1\le i\ne j\le 2^n-1$; then, since $M_i, M_j$ are distinct maximal subgroups of $H \cong \mathcal C_2^n$, we have 
$$M_iM_j=2^{n-2}H,$$
so that for $1\le i\ne j\le 2^n-1$ we have
$$M_i(H-M_j)=2^{n-1}H-2^{n-2}H=2^{n-2}H.$$

We  use this to obtain:
\begin{align*}
D\cdot D^{-1}&=\left (\sum_{i=1}^{2^n-1} a_{S_i}M_i\right) \left (\sum_{i=1}^{2^n-1}a_{S_i}(H-M_i)\right )\\
&=\sum_{1\le i \ne j\le n}^{2^n-1}a_{S_i}M_ia_{S_j}(H-M_j) +
\sum_{1\le i \le n}^{2^n-1}a_{S_i}^2M_i(H-M_i)\\
&=2^{n-2}\sum_{1\le i \ne j\le n}^{2^n-1}a_{S_i}a_{S_j}H+
\sum_{1\le i \le n}^{2^n-1}a_{S_i}^2(2^{n-1}H-2^{n-1}M_i)\\
   &\tag {7.3}\label {e73}  =2^{n-2}\sum_{1\le i \ne j\le n}^{2^n-1}a_{S_i}a_{S_j}H+2^{n-1}
\sum_{1\le i \le n}^{2^n-1}a_{S_i}^2(H-M_i)
\end{align*}
   
   Since $|\mathfrak G_{n,k}|=2^{2n}, h=|H|=2^n$ we have
   \begin{align*}&
   k=\frac {2^{2n}-2^n}{2}=2^{n-1}(2^n-1),\quad 
\lambda= 2^{n-1}(2^{n-1}-1).
 \end{align*}   
   
   Returning to equation (\ref {e73}), in particular looking at the first sum of equation (\ref {e73}),
    we see that every non-trivial coset of $H$ occurs $2^n-2$ times in equation (\ref {e73}).
     Thus from equation (\ref {e73}) we see that the coefficient in $DD^{-1}$ of each element of that 
   coset is 
   $$2^{n-2}(2^n-2)=2^{n-1}(2^{n-1}-1)=\lambda,
   $$ as we desire.
   
   The second sum of equation (\ref {e73}) gives the contributions to the trivial $H$-coset. We rewrite it as
\begin{align*}\tag {7.4} \label {e74} 2^{n-1}
\sum_{1\le i \le n}^{2^n-1}a_{S_i}^2(H-M_i)=
   2^{n-1}
\sum_{1\le i \le n}^{2^n-1}(H-a_{S_i}^2M_i).
  \end{align*}
   But we are assuming that $a_{S_i}^2\notin M_i$, so we must have
   $H-a_{S_i}^2M_i=M_i.$ Thus equation   (\ref {e74}) is
   \begin{align*}
\tag {7.5} \label {e75}       2^{n-1}
\sum_{1\le i \le n}^{2^n-1}M_i.
  \end{align*}
  
  Now since the $M_i$ are distinct maximal subgroups, and there are $2^n-1$ of them,  we see that every maximal subgroup of $H\cong \mathcal C_2^n$
   is in the list $M_1,\dots, M_{2^n-1}$, and so one  has
  $$\sum_{1\le i \le 2^n-1}M_i=(2^n-1)\cdot 1 +(2^{n-1}-1)(H-1).$$
  
  Thus if $h' \in H, h' \ne 1,$ then the coefficient of $h'$ in equation equation (\ref {e75}) is 
  $$2^{n-1}(2^{n-1}-1)=\lambda,$$ as required. 
  
 The coefficient of $1$ in $D \cdot D^{-1}$ is then
 $$k^2-\lambda(|\mathfrak G_{n,k}|-1)=2^{2n-2}(2^{n-1}-1)^2-2^{n-1}(2^{n-1}-1)(2^{2n}-1),$$ which is equal to $k$, as required.
 Thus we have $D \cdot D^{-1}=\lambda (G-1)+k\cdot 1.$
 \qed\medskip

 \section {Another criterion for the existence of a relative skew Hadamard difference set group}

\begin{thm}\label{thmseth}
Given $H \lhd G$, $H \cong  \mathcal{C}_2^h, [G:H]=h,$ let $t_1 = 1, t_2, \dots, t_h$ be a transversal for $H$ in $G$. Then for each $1 \leq i \leq h$ there is a $j = j(i)$ such that $t_i^{-1} \in Ht_j$. 
Suppose there exist distinct maximal subgroups $H_2, \dots , H_h$ of $H$ such that

\noindent (1)  $(H_j)^{t_i^{-1}} = H_i,$ and 

\noindent (2) $t_it_j \in H_i$,

\noindent for all $2 \leq i \leq h$ where $j = j(i)$. Then $D = \sum_{i=2}^hH_it_i$ is a difference set such that $D = D^{-1}$.
\end{thm}
\begin{proof} We note that $H_2,\dots,H_h$ are all the maximal subgroups of $H$.
First, to show that $D = D^{-1}$, we show that the identity occurs $k$ times in the product $D^2$:
\begin{align*}
D^2 &= \left(\sum_{i=2}^hH_it_i \right) \left(\sum_{\ell=2}^hH_{\ell}t_{\ell} \right) = \sum_{i=2}^h \sum_{\ell=2}^h H_it_iH_{\ell}t_{\ell}. 
\end{align*}
Fix $2\le i\le h$ and let $j=j(i)$.
If $\ell \ne j$, then  $Ht_it_\ell \ne H$ so that $1 \notin H_it_iH_\ell t_\ell$. If $\ell=j$, then using (1) and (2) we have 
\begin{align*}
H_it_i(H_{j}t_{j}) &= H_i(t_iH_{j}t_i^{-1})t_it_{j}  \\
&= H_i^2t_it_j 
= \frac{h}{2} H_i 
\end{align*}

There are $h-1$ values of $i$, so   the identity occurs $\frac{h}{2}(h-1) = k$ times in $D^2$. Thus, $D = D^{-1}$. 

 Now we will show that $D \cdot D^{-1} = k + \lambda(G - 1)$. We have just shown that the identity occurs $k$ times in $D\cdot D^{-1}=D^2$.
Consider the product
\begin{align*}
D \cdot D^{-1} &=D\cdot D= \sum_{i=2}^h \sum_{\ell=2}^h H_it_iH_{\ell}t_{\ell}.
\end{align*}
When $t_it_j \in H$, we have, as above, 
\begin{align*}
H_it_iH_{j}t_{j} &= \frac{h}{2} H_i 
\end{align*}

 Each element of $H$ that is not the identity occurs in $\frac{h - 2}{2}$ of the $h - 1$ maximal subgroups $H_i,2\le i\le h,$ of $H\cong \mathcal C_2^n$. Thus, the coefficient, in $D^2$, of each element of $H$ that isn't the identity is $\frac{h}{2}(\frac{h - 2}{2}) = \lambda$.

Now, when $t_it_\ell \notin H$, we have, 
\begin{align*}
H_it_i(H_{\ell}t_{\ell}) &= H_i(t_iH_{\ell}t_i^{-1})t_it_{\ell} 
\end{align*}
We  observe that conjugation by $t_i^{-1}$ permutes $H_2,H_3,\dots,H_h$. Thus, since $t_iH_jt_i^{-1} = H_i$, and $j(i) \neq \ell$, we have that $t_iH_{\ell}t_i^{-1} \neq H_i$, so 
\begin{align*}
H_it_iH_{\ell}t_{\ell} &= H_i(t_iH_{\ell}t_i^{-1})t_it_{\ell} \\
&= \frac{h}{4}Ht_it_{\ell}.
\end{align*}

For $g \notin H$, each non-trivial coset $Hg$ occurs  $h-2$ times in the form $Ht_it_\ell$. So  $g$ will occur $\frac{h}{4} (h - 2) = \lambda$ times in $D^2$. Thus 
\begin{align*}
D \cdot D^{-1} &= k + \lambda(G - 1),
\end{align*}
which shows that $D$ is a difference set.
\end{proof}

\noindent {\bf Remark} A computer calculation \cite {Ma} shows that, of the $267$ groups of order $64$, at least $50$ of them satisfy the hypotheses of Theorem \ref {thmseth}. These include groups with  nilpotency classes $1,2,3,4$. There are also four groups of order $16$ that satisfy the hypotheses of Theorem \ref {thmseth}. 
\medskip

\begin{prop} \label{propex2} The group $G=\mathcal C_4^n$ satisfies the hypotheses of Theorem \ref {thmseth}.
\end{prop}
\noindent {\it Proof} Let $a_1,a_2,\dots,a_n\in G$ be the generators of order $4$, and let $b_i=a_i^2, i=1,\dots,n$. Let $H=\langle b_1,b_2,\dots,b_n\rangle$. 
Let $t_0=1,t_1,t_2,\dots,t_{2^n-1}$ be a transversal for $H$.
We clearly have hypothesis (1) of Theorem \ref {thmseth}, since $G$ is abelian and $j(i)=i$. Thus we must order the maximal subgroups $M_1,M_2,\dots, M_{2^n-1}$ of $H$ so that we have (2): $t_i^2 \in M_i$ for $1\le i\le 2^n-1$. Note that $t_1^2,t_2^2,\dots,t_{2^n-1}^2$ are all the non-trivial elements of $H$. 

Now we think of $H$ as the additive group $V=\mathbb F_2^n$, so that maximal subgroups correspond to maximal subspaces of $V$. Further, 
each such maximal subspace $M$  is determined by a vector that we denote $v_M$, so that  we have
$M=\{v \in V:v\cdot v_M=0\}$.   Here $\cdot$ denotes the standard dot product on $V$, taking values in $\mathbb F_2$.

 Each element $t_i^2$ is then a vector that we denote by $w_i$, and we need to find a correspondence $w_i \leftrightarrow v_{M_{i}}$ where $w_i \in M_i$, i.e. where $v_i\cdot v_{M_i}=0$   This is thus equivalent to finding a bijection  
$\kappa:V^\times \to V^\times$ such that $v\cdot\kappa(v)=0$ for all $v \in V, v \ne 0$. We show how to do this. In fact we will define $\kappa$ to be an involution. 

Let $(1) \in V$ denote the all $1$ vector. 

Assume first that $n$ is even, and let
$$\kappa(v)=(1)-v, \text { for } v \ne 0,(1), \quad \kappa((1))=(1).$$
Since $n$ is even we have $(1)\cdot (1)=0$, while if $v \ne (1)$, then $v\cdot \kappa(v)=0$. This does the $n$ even case.

Now assume that $n$ is odd. Let
$$S=\{0,(1,0,0,\dots,0), (1,1,0,0,\dots,0),(0,1,1,\dots,1),(0,0,1,1,\dots,1),(1)\}\subset V.$$
Then for $v \in V \setminus  S$ we let   $\kappa(v)=(1)-v,$ noting that $(1)-v \in V \setminus S$.  We define $\kappa$ on the elements of $S$ as follows: 
\begin{align*}&
\kappa((1))=(1,1,0,\dots,0),\quad \kappa((1,0,\dots,0))=(0,0,1,1,\dots,1),\\&
\kappa((0,1,1,\dots,1,1))=(0,1,1,\dots,1,1).
\end{align*}
This defines $\kappa$ (as an involution) and one checks that, since $n$ is odd, this gives $v\cdot \kappa(v)=0$ for a $v \ne 0$.\qed\medskip

\section {Normal subgroups}

Suppose that $G$ is a relative skew Hadamard difference set with difference set $D$ and subgroup $H, h=|H|, h^2=|G|$.

Let $N\le  G, n=|N|,$ be normal, with $\pi:G \to K=G/N$ the quotient map.
Let $g_0=1,g_1,\dots,g_{u-1}$ be coset representatives for $K=G/N, u=|K|$. Then $\pi(G)=\frac {h^2} {u} K$.

Let $$x_i=|D \cap Ng_i|,\quad y_i=|H \cap Ng_i|,\quad  1 \le i \le u,$$ so that we can write 
$$\pi(D)=\sum_{i=1}^{u} x_ig_i, \quad \pi(H)=\sum_{i=1}^{u} y_ig_i.$$ 
It follows that
$$\pi(D^{-1})=\sum_{i=0}^{u-1} x_ig_i^{-1}.$$ 
Then from $D \cdot D^{-1}=\lambda G+(k-\lambda)1$ we get
\begin{align*} 
\tag {9.1}\label {t91}
\left (  \sum_{i=0}^{u-1} x_ig_i \right )\cdot \left (   \sum_{i=0}^{u-1} x_ig_i^{-1} \right ) =\lambda n \left ( \sum_{i=0}^{u-1} g_i \right ) + (k-\lambda) 1_K.
\end{align*}
We also have
\begin{align*}
\tag {9.2}\label {t92}  \sum_{i=0}^{u-1} x_i=k, \quad \sum_{i=0}^{u-1} y_i=h.\end{align*}
From $DH=\frac h 2 (G-H)$ we obtain
\begin{align*}\tag {9.3}\label{t93} 
\left (  \sum_{i=0}^{u-1} x_ig_i \right )\cdot \left (   \sum_{i=0}^{u-1} y_ig_i \right ) =\frac h 2 \left (\frac {h^2} u K  -   \left (   \sum_{i=0}^{u-1} y_ig_i \right ) \right ).
\end{align*}
Lastly, from $G=D+D^{-1}+H$ we get
\begin{align*}\tag {9.4}
\label {t94} \frac {h^2} u K = \sum_{i=0}^{u-1} x_ig_i +\sum_{i=0}^{u-1} x_ig_i ^{-1}+ \sum_{i=0}^{u-1} y_ig_i.
 \end{align*}

Each of equations (\ref {t91})-(\ref {t94}) determines $u=\frac {h^2} {n}$ elements of 
$R=\mathbb Q[x_1,\dots,x_u,$ $y_1,\dots,y_u]$, one for each element of $K$. 

Let $\mathcal I_K\subseteq R$ denote the ideal generated by the elements in equations (\ref {t91})-(\ref {t94}).\medskip 

\noindent {\bf Example 9.1: $K=\mathbb C_2^2$.} So suppose that $K=\mathcal C_2^2$. Then finding a Gr\"obner basis for $\mathcal I_K$ gives
\begin{align*}&
x_0=\frac 1 8 h^2-\frac h 2 ,\,\,x_1=x_2=x_3=\frac {1} 8 h^2;
y_0=h,y_1= y_2=y_3=y_4=y_5=0.
\end{align*} Since $y_0=h$ this shows that $H \le N$.
\medskip

A slightly more involved example is given by

\noindent {\bf Example 9.2: $K=\mathcal C_4$.} So suppose that $K=\mathcal C_4=\langle t\rangle,$ where the ordering of the elements is $g_1=1,g_2=t,g_3=t^2,g_4=t^3$. Then finding a Gr\"obner basis for $\mathcal I_K$ gives:
\begin{align*}&
x_0 -  \frac 1 2 y_2 -  \frac 1 8 h^2 +  \frac 1 2 h,\,\,\,
    x_1 + x_3 -  \frac 1 4 h^2,\,\,\,
    x_2 +  \frac 1 2 y_2 -  \frac 1 8 h^2,\\&
    x_3^2 -  \frac 1 4 h^2 x_3 -  \frac 1 8 h y_2 +  \frac 1 {64} h^4,\,\,\,
    x_3 y_2 -  \frac 1 2 h x_3 -  \frac 1 8 h^2 y_2 +  \frac 1 {16} h^3,
    y_0 + y_2 - h,\\&
    y_1,
    y_2^2 -  \frac 1 2 h y_2,
    y_3.
\end{align*}  This ideal has dimension zero. 
The fifth element of this basis factors as
$$\left (y_2 - \frac 1 2 h\right )\left (
    x_3 - \frac 1 8 h^2\right ),$$ and the eighth as $y_2  \left (y_2 - \frac 1 2 h\right)$.
By an elimination process one finds that the variable $x_3$ satisfies
$$\left (x_3 - \frac 1 8 h^2\right )\left (x_3 - \frac 1 8 h^2 - \frac 1 4 h\right )\left (x_3 - \frac 1 8 h^2 + \frac 1 4 h\right )=0.$$
Thus one of the factors in each of these three  products must be zero.
This gives $12$ cases to check, and a quick analysis shows that
 we have one of
\begin{align*}&
(i) \,\, x_0=x_1=x_2=\frac h 4\left  (\frac h 2 -1\right ),\,\,x_3=\frac h 4\left  (\frac h 2 +1\right ),
y_0=y_2=\frac h 2,\,\, y_1=y_3=0;\\&
(ii) \,\, x_0=x_2=x_3=\frac h 4\left  (\frac h 2 -1\right ),\,\,x_1=\frac h 4\left  (\frac h 2 +1\right ),
y_0=y_2=\frac h 2,\,\, y_1=y_3=0;\\
&(iii) \,\, x_0=\frac 1 8 h^2-\frac 1 2 h,\,\,x_1=x_2=x_3=\frac 1 8 h^2,\,\,
y_0=h, y_1= y_2=y_3=0.
\end{align*}
\medskip

\noindent{\bf Example 9.3: $K=\mathcal C_6$}:
For $G=\mathcal C_6=\langle t\rangle$ a similar argument shows that we have 
\begin{align*}&
 x_0=\frac 1 {12} h^2-\frac 1 2 h,\,\,x_1=x_2=x_3=x_4=x_5=\frac {1} {12} h^2;\\&
y_0=h,y_1= y_2=y_3=y_4=y_5=0.
\end{align*} Since $y_0=h$ we have $H \le N$.
\medskip

\noindent{\bf Example 9.4: $K=\mathcal C_9=\langle t \rangle$}:  One further idea occurs in this case: Calculating  $\mathcal I_K$ and doing an elimination shows that
the variable $y_1$ (corresponding to the element $t$) satisfies
\begin{align*}&
 y_1
  (  y_1 + \frac 1 9  (-\zeta ^{10} + \zeta ^4 + \zeta ^2 - 1) h)
  (    y_1 + \frac 1 9  (\zeta ^8 - \zeta ^4 - \zeta ^2 - 1) h)\times \\&
    (  y_1 + \frac 1 9  (\zeta ^{10} - \zeta ^8 - 1) h)
    (  y_1 + \frac 1 9  (\zeta ^{10} - \zeta ^4 - \zeta ^2 - 1) h)
    (  y_1 + \frac 1 9  (-\zeta ^{10} + \zeta ^8) h)\times\\&
    (  y_1 + \frac 1 9  (-\zeta ^{10} + \zeta ^8 - 1) h)
    (  y_1 + \frac 1 9  (\zeta ^{10} - \zeta ^4 - \zeta ^2) h)
    (  y_1 + \frac 1 9  (-\zeta ^8 + \zeta ^4 + \zeta ^2) h)\times\\&
    (  y_1 - \frac 1 9  h)
    (  y_1 + \frac 1 9  (-\zeta ^8 + \zeta ^4 + \zeta ^2 - 1) h)
    (  y_1 + \frac 1 9  (\zeta ^8 - \zeta ^4 - \zeta ^2) h)\times\\&
    (  y_1 + \frac 1 9  (-\zeta ^{10} + \zeta ^4 + \zeta ^2) h)
     ( y_1 + \frac 1 9  (\zeta ^{10} - \zeta ^8) h=0.
\end{align*}
Here $\zeta=\zeta_{36}$ is a primitive $36$th root of unity. Thus we see that $y_1$ has to be a root of one of these linear  factors.
However we know that $y_1 \in \mathbb Z$, so that the only possibilities are $y_1=0, y_1=\frac h 9$, thus greatly restricting the number of cases that we have to consider. One can apply this idea to many of the variables $x_i,y_j$ so as to reduce the number of cases that one has to consider. In this way we conclude that
$y_0=h$.
\medskip



One can use the above techniques to  show:
\begin{prop} \label{prop99}If
$N \triangleleft G$, where $G/N$ is one of 
$$\mathcal C_2^2, \mathcal C_2^3, S_3, \mathcal C_6, \mathcal C_9,  \mathcal C_3^2,   \mathcal C_{10}, D_{10},
\mathcal C_2 \times S_3, \mathcal C_2 \times \mathcal C_6,$$ then $H \le N$. \qed\end{prop}


  \section{Representations}
Suppose that $G$ is a relative skew Hadamard difference set with difference set $D$ and subgroup $H, h=|H|$.
  We recall that $D, D^{-1},G,H$ satisfy the equations
\begin{align*}
&(10.1) \,\, D^2=\lambda G+\frac h 2 H-(k-\lambda);\quad
(10.2) \,\, DD^{-1}=\lambda G+(k-\lambda);\\&
(10.3)\,\, HD=\frac h 2 (G-H).
\end{align*}

Let $\rho$ be a non-principal irreducible representation of $G$ with irreducible character $\chi$ and $d=\chi(1)$. We assume that $\rho$ is unitary.  Since $\chi$ is not principal we see from orthogonality of the character table that  $$\chi(G)=0.$$ Since $\rho$ is unitary we see that $\rho(D^{-1})=D^*$.  Now we know from Lemma \ref {lemdiag}
that $D,D^{-1},G, H$ pairwise commute, and so $\{\rho(D),\rho(D^{-1}),\rho(H),\rho(G)\}$ is a set of commuting normal matrices. Thus they are simultaneously diagonalizable, and we may assume that in fact they are diagonal matrices. 

Since $H \triangleleft G$ and $\rho$ is irreducible  it follows that $\rho(H)$ is a scalar matrix, which we write as 
$$\rho(H)=h_0\rho(1), \,\,\, h_0 \in \mathbb C.$$
  Since $H^2=hH$ we have $\rho(H)^2=h\rho(H)$, which shows that either $\rho(H)=0$ or $\rho(H)=h$; i.e. $h_0 \in \{0,h\}$. 
From (10.1) and (10.2) we see that 
$$
\rho(D)^2=\frac {h(2h_0-h)} 4\rho(1);\quad 
\rho(D)\rho(D)^*=(k-\lambda)\rho(1)=\frac {h^2} 4 \rho(1).$$

\noindent CASE 1: If $h_0=0$, then these give (where $i^2=-1$) 
$$\rho(D)=\mathrm{diag}\left (\varepsilon_1 i\frac h 2,\varepsilon_2  i\frac h 2,\dots,\varepsilon_d  i\frac h 2\right ).$$
Here $\varepsilon_i \in \{-1,1\}$.  In this case $\mu(\rho(D))$ divides $x^2+\frac {h^2} 4.$
\medskip 

\noindent CASE 2: If $h_0=h$, then (10.3) gives
$$h\rho(D) =-\frac {h^2}{2}\rho(1),
\text { so that } 
\rho(D)=-\frac h 2 \rho(1).
$$ But then we   have $\rho(D^{-1})=D^*=D$.
In this case $\mu(\rho(D))=x+\frac {h} 2.$
This gives Theorem \ref {thmrep}.\qed 




      
      \section {Some tests for relative skew Hadamard difference set groups} 
      
      We will say that a group $K$ is {\it $H$-swallowing}, if whenever we have a relative skew Hadamard difference set group $G$ with subgroup $H$ and there is a normal subgroup $N$ of $G$ with $G/N \cong K$, then $H \le N$.
      
    For a group $G$ let $\mathcal N(G)$ denote the intersection of all normal subgroups of $G$ that have prime index or whose quotients are $H$-swallowing. 
    Thus if $G$ is a $p$-group, then $\mathcal N(G)$ is just the Frattini subgroup of $G$. 
    Then by Theorem \ref {thmfr} we see that $H \le \mathcal N(G)$.  Thus we   have our first test:
    
   \noindent    (T1) $|\mathcal N(G)| \equiv 0 \mod h$. 
    
    Other tests that we can use are (see Theorem \ref {thm1}):
    
    \noindent   (T2) $H$ contains  the subgroup generated by all the involutions.
    
   \noindent    (T3) $G$ contains a normal  subgroup of order $h$ and index $h$.
    
    \noindent    (T4) $H$ does not have a complement in $G$.
    
    Using tests (T1), (T2), (T3), (T4)  enables us to eliminate  various groups from being relative skew Hadamard difference sets groups:
    
    \noindent {\bf The case  $h=6$} Of the $14$ groups of order $36$, only  the group $G_{36,1}$ passes the above tests. One can show by a short computer calculation that $G_{36,1}$  is not a relative skew Hadamard difference set group. 
    Here 
    \begin{align*}G_{36,1}&=\langle a,b,c,d|
    a^2 = b,
    b^2 = 1,
    c^3 = d^2,
    d^3 = 1,\\&\qquad 
    b^a = b,
    c^a = c^2  d,
    c^b = c,
    d^a = d^2,
    d^b = d,
    d^c = d\rangle.\end{align*}
This gives:
\begin{thm} \label{thm36} For $h=6$ there are no relative skew Hadamard difference set groups.
\qed\end{thm}

    
        


    
         
   \section{Can we have $m>0$?}
   
   Let $G$ be a relative skew Hadamard difference set group with subgroup $H$ and difference set $D$. Assume, as in $\S 1$, that 
    (1) $D \cap D^{-1}=Hg_1\cup \dots \cup Hg_m$;
 (2) $G \setminus (D \cup D^{-1})=H \cup Hg_1'\cup \dots \cup Hg_m'$.

   We note that $0\le m \le (h-1)/2,$ but we can do better: 
   \begin{lem}\label{lemm} $m \le \frac {h-1} 4.$
  \end{lem}
   \noindent{\it Proof}
      Let $\pi:G \to Q=G/H$ be the quotient map. Then $|Q|=h, \pi(G)=hQ, \pi(H)=h$. Let $Q=\{q_0,q_1,\dots,q_{h-1}\}$ and put 
   $$\pi(D)=\sum_{i=0}^{h-1} x_iq_i, \text { so that } 
   \pi(D^{-1})=\sum_{i=0}^{h-1} x_iq_i^{-1}.$$
   
   Now we know that $\sum_{i=0}^{h-1} x_i=k$ and  by Theorem 7.1 of \cite {mp} we have: $$\sum_{i=0}^{h-1} x_i^2=k+\lambda(h-1)=h^2(h-1)/4.$$ 
   But $D$ contains $m$ cosets of $H$ and so $\sum_{i=0}^{h-1} x_i^2\ge mh^2.$ This gives the result.\qed\medskip

   \begin{cor} \label{lemm468}
   For $h=4$ we have $m=0$ and for $h=6,8$ we can only have $m=0,1$.\qed 
 \end{cor}\medskip 


\section {Open questions}

1. Are there relative skew Hadamard difference set groups that are not $2$-groups.

2. Are there relative skew Hadamard difference set groups with $m>0$?

 \end{document}